\documentclass[12pt]{article}

\usepackage{amsfonts}
\usepackage{amsmath}
\usepackage{amssymb}
\usepackage{amsthm}

\def\calE{\mathcal{E}}
\def\calF{\mathcal{F}}
\def\calT{\mathcal{T}}

\def\Pc{P_{cr}}

\def\ra{\rightarrow}

\def\lan{\langle}
\def\ran{\rangle}
\def \Z {\mathbb Z}
\def \I {\mathbb I}
\def \R {\mathbb R}

\def\al{\alpha}

\def\ep{\varepsilon}

\def\la{\lambda}

\def\ga{\gamma}

\newtheorem{thm}{Theorem}[section]
\newtheorem{prop}[thm]{Proposition}

\theoremstyle{plain}

\newtheorem{example}[thm]{Example}

\newtheorem{claim}[thm]{Claim}

\begin{document}

\title{Random spatial growth \\ with paralyzing obstacles}

\author{J. van den Berg\footnote{Research funded in part by the Dutch BSIK/BRICKS project.}, \,
Y. Peres\footnote{Research supported in part by NSF grant DMS-0605166.}, \,
V. Sidoravicius\footnote{Partially supported by CNPq, Brazil}\,  and
M.E. Vares \footnote{Partially supported by CNPq, Brazil}
  \\
{\small CWI and VUA, Microsoft and  UC Berkeley, IMPA and CBPF} \\
{\footnotesize email: J.van.den.Berg@cwi.nl;
peres@stat.berkeley.edu; vladas@impa.br; eulalia@cbpf.br} }
\date{}
\maketitle

\begin{abstract}
We study models of spatial growth processes where initially there
are sources of growth (indicated by the colour green) and sources of
a growth-stopping (paralyzing) substance (indicated by red). The
green sources expand and may merge with others (there is no
`inter-green' competition). The red substance remains passive as
long as it is isolated. However, when a green cluster comes in touch
with the red substance, it is immediately invaded by the latter,
stops growing and starts to act as red substance itself. In our main
model space is represented by a graph, of which initially each
vertex is randomly green, red or white (vacant), and the growth of
the green clusters is similar to that in first-passage percolation.
The main issues we investigate are whether the model is well-defined
on an infinite graph (e.g. the $d$-dimensional cubic lattice),
and what can be said
about the distribution of the size of a green cluster just before it
is paralyzed. We show that, if
the initial density of red vertices is positive, and that of white
vertices is sufficiently small, the model is indeed well-defined and
the above distribution has an exponential tail. In fact, we believe
this to be true whenever the initial density of red is positive. \\
This research also led to a relation between invasion percolation and critical
Bernoulli percolation which seems to be of independent interest.
\end{abstract}

\noindent {\it 2000 MSC:} primary 60K35, secondary 60K37, 82B43.

\noindent {\it Key words and phrases:} Growth process,
percolation, invasion percolation.

\begin{section}{Introduction}
\begin{subsection}{Description of the model and the main problems}
Consider the following model where different `objects' (or
`populations') grow simultaneously until they hit a paralyzing
substance, in which case they stop growing and become paralyzing
themselves: Each vertex of a connected, finite (or countably
infinite, locally finite) graph $G= (V,E)$ is initially,
independently of the other vertices, white, red or green with
probabilities $p_w$, $p_r$ and $p_g$ respectively. Each edge of $G$ is
initially closed. By a green cluster we will mean a maximal
connected subgraph of $G$ of which all vertices are green and all
edges are open. We denote the green cluster containing $v$ at time
$t$ by $C_g(v,t)$. (If $v$ is not green at time $t$, then $C_g(v,t)$
is empty). It is clear from the above that initially the only green
clusters are single green vertices. These green clusters can
grow, merge with other green clusters and finally become paralyzed (red) as follows. \\
Whenever an edge $e = \lan v,w \ran$ is closed and has at least one
green end-vertex, say $v$, it becomes open at rate $1$. Moreover,
immediately after it gets open the following action takes place
instantaneously: If exactly one end-vertex, say $v$, is green and
the other, $w$, is white, $w$ becomes green (and we say, informally,
that the green cluster of $v$ grows by absorbing $w$). If $w$ is
red, then each vertex in the green cluster of $v$ becomes red (and
we say that the green cluster of $v$ becomes paralyzed). Finally, if
both vertices are green, no extra action takes place. (Note that in
this case the two vertices may have been in two different green
clusters right before the opening of $e$, but are now in the same
green cluster).

Note that once an edge is open it remains open, that once a vertex is green it never turns white (but may become red),
and once a vertex is red it remains red.

\smallskip\noindent
Let us first consider the case where the graph $G$ is finite. In
that case the above process is clearly well-defined and has some
obvious properties, which we will state after introducing the
following terminology. By a configuration (or `site-bond
configuration') we mean an element of $\{0,1\}^E \, \times \,
\{\mbox{ green, red, white } \}^V$, where $0$ and $1$ denote `open'
and `closed' respectively. An `open-bond cluster' (with respect to a
configuration) is a maximal connected subgraph of $G$ of which all
edges are open (for that configuration). We say that it is
non-trivial if it has at least one edge. Note that the earlier
defined `green cluster' is an open-bond cluster of which each vertex
is green. A `red cluster' is defined similarly. We call a
configuration admissible if each non-trivial open-bond cluster is
either a red cluster or a green cluster. Now we are ready to state
the announced simple properties and observations: If $G$ is finite,
the process is a Markov chain on the set of admissible
configurations. The admissible configurations where no vertices are
green or all vertices are green are absorbing, and the chain will
with probability 1 end in one of those configurations. In
particular, if initially there was at least one red vertex, then every
green vertex will eventually become red. Moreover (because initially
all edges were closed) at any time, every non-empty red cluster $\mathcal{C}$
contains exactly one vertex $v$ that was originally red.
We say that this vertex $v$ is `responsible for' the other vertices in
$\mathcal{C}$  becoming red (or, that the vertices in $\mathcal{C}$ became red
`due to' $v$).

\smallskip\noindent
If $G$ is {\it infinite}, for instance the $d$-dimensional cubic
lattice, the situation is much more problematic, due to the fact
that the range of the interaction is not bounded: an entire
cluster, no matter how large, can change colour instantaneously.
The main questions we address in this paper concerning the above
process, and some other related models, are:

\begin{itemize}
\item {\bf 1.} Does the dynamics exist? This is a nontrivial issue for
such interacting processes on infinite graphs: See for instance,
Aldous' frozen percolation process (\cite{A}), which was shown by
Benjamini and Schramm (1999, private communication) not to exist in
$\Z^2$. For related matters on the non-existence of that process,
see also Remark (i) in Section 3 of \cite{BeT} and the example due
to Antal J\'arai (1999, private communication) which follows it. A
crucial difference between Aldous' model and ours is that in Aldous'
model, clusters freeze only when they are infinite, while we believe that in our model, due to the
positive density of initially red vertices, the green clusters do
not become infinite (see the next item).
A model which has more in common with ours is
the forest-fire model studied in \cite{D}. But again there is a
major difference: in that model there is a uniform lower bound for
the probability that a cluster of interest is `destroyed' before
growing further, and this uniform bound is a crucial ingredient in
the existence proof in \cite{D}. In our model there seems to be no
 analog of such a property.

\item{\bf 2.} Is a green cluster
      always finite at the moment it becomes red? Does the distribution of its radius (and of its volume)
      have an exponential tail?
\item {\bf 3.} Let $w$ be an originally red vertex. Is the set of
originally green vertices $v$ with the property that $w$ is
responsible for $v$ becoming red, finite? Does the distribution of
its volume have an exponential tail?
\end{itemize}

The organization of the paper is as follows. In Subsection 1.2 we
give a partial answer to the questions listed above. In particular,
Theorem \ref{mainthm} states that, for $G = \Z^d$ and $p_w$
sufficiently small, the answers to the above questions are positive.
Our research also led to a new result for invasion percolation (see
Theorem \ref{uni-bound} and Proposition \ref{inv-perc}). In
Subsection 1.3 we explain the notion of `autonomous region' which
plays an important role in this paper. In subsection 1.4 we briefly
discuss some alternative versions of the model. In section 2 we give
a proof of the main result for the special case where $p_w = 0$. It
turns out that that case can be dealt with in a very elegant and
transparent way. It serves as an introduction to the proof of the
more complicated case where $p_w$ is small but positive, which is
given in Section 3. At the end of Section 3 we come briefly back to
the alternative versions of the model discussed in Subsection 1.4.
\end{subsection}

\begin{subsection}{Statement of the main results}
Let $G$ be a connected, countably infinite graph of bounded degree,
and consider the model presented in Subsection 1.1, with parameters
$p_w$, $p_g$ and $p_r$. Our main result, Theorem \ref{mainthm}
below, states, among other things, that under certain conditions the
dynamics is well-defined. The formulation of the condition requires
some additional notation and terminology: By the distance $d(v,w)$
between two vertices $v$ and $w$ of $G$ we mean the length (i.e.
number of edges) of the shortest path from $v$ to $w$. The diameter
of a set of vertices $W$ of $G$ is defined as $\max_{v, w \in W}
d(v,w)$, and $\partial W$ will denote the set of all vertices that
are not in $W$ but have an edge with some vertex in $W$.
The number of elements of a set $W$ will be denoted by
$|W|$. For a finite graph $H$, denote by $|H|$ the number of
vertices in $H$. Let $D$ denote the maximal degree in $G$.
% i.e. the maximum
%number of edges incident to any given vertex.

For each vertex $v$ of $G$ and $p \in (0,1)$, let $\xi_v(p)$ denote
the expectation of the volume (i.e. number of vertices) of the
occupied cluster of $v$ in site percolation on $G$ with
parameter $p$. Further, define
$$\xi(p) = \sup_v \xi_v(p).$$
Recall the definition of $C_g(v,t)$ in Subsection 1.1.
We are now ready to state our main results.
\begin{thm}\label{mainthm}
Suppose that
\begin{equation} \label{key}
(D-1) \xi(p_w) < p_r \,.
\end{equation}
We have
%\begin{itemize}
\item
%Every site is contained in an autonomous region; consequently
{(a)} The dynamics on $G$ is well-defined. With probability 1, at
any time, each red cluster has a unique initially red vertex.

\item (b) For any originally green
vertex $v$, let $C_g(v) = \cup_{t \geq 0} C_g(v,t)$ be the green
cluster of $v$ just before it becomes red. Let $|C_g(v)|$
be the number of vertices of $C_g(v)$. Then, with probability $1$,
$|C_g(v)|$ is finite for each such $v$. Moreover, the distribution of
$|C_g(v)|$ has an exponential tail.

\item (c) If $G$ is a Cayley graph and $w$ is an originally red vertex in $G$,
then the set $D(w)$ consisting of all green vertices that become red
due to $w$ is finite; moreover, the diameter of $D(w)$ has an
exponential tail. (Here, extending the definition given before in
the case of finite $G$, if $v$ is an originally green vertex and $w$
is the (unique a.s.) originally red vertex in the red clusters that
eventually contain $v$, we say that $v$ becomes red due to $w$.)

\item (d) If $G$ is the $d$-dimensional cubic lattice, then the distribution of
$|D(w)|$ also has an exponential tail.

%\end{itemize}
\end{thm}

Note that in the case $p_w = 0$, condition~\eqref{key} of Theorem
\ref{mainthm} is satisfied for every positive $p_r$. For this case
we have, in addition to Theorem \ref{mainthm}, considerably
stronger results. In particular, the
following theorem holds, where we fix $p_w = 0$ and then vary the
parameter $p_r$. In this theorem and its proof, $P_p$ denotes the
ordinary (Bernoulli) bond percolation measure with parameter $p$
and $\Pc$ stands for $P_{p_c}$, where $p_c$ denotes the critical
probability for this percolation model. By $B(n)$ we denote the
set of all vertices at (graph) distance $\leq n$ from some
specified vertex $O$. The event that there is an open path from $O$ to $\partial B(n)$ is
denoted by $\{O \leftrightarrow \partial B(n)\}$. Further, the symbol $\approx$ denotes
logarithmic equivalence, i.e., we say for two positive functions
$g(n)$ and $h(n)$ that $g(n) \approx h(n)$ as $n \rightarrow
\infty,$ if

$$\frac{\log h(n)}{\log g(n)} \rightarrow 1, \,\,\, n \rightarrow \infty.$$
Let $W$ be a set of vertices in a graph $G$ with a distinguished vertex $O$.
By the {\em radius} of  $W$ we mean the maximal distance from $O$ to a vertex of $W$.
We are now ready to state the following theorem.

\begin{thm} \label{uni-bound}
Let $C_g(\cdot)$ be as in part (b) of Theorem \ref{mainthm}. If $G$
is the square lattice in two dimensions (or the triangular or the
hexagonal lattice), and $p_w = 0$, then
$$P(\mbox{The radius of } C_g(O) \mbox{ is at least } n) \,\, \uparrow f(n), \,\,
\mbox{ as } p_r \downarrow 0,$$
where $f$ is a function satisfying

$$f(n) \approx \Pc(O \leftrightarrow \partial B(n)).$$

\end{thm}

Theorem \ref{uni-bound} follows easily from the following
Proposition concerning invasion percolation on the
lattices considered in the theorem. Before we state it, we briefly
recall the invasion percolation model (on these lattices) and some
of its basic properties. (Invasion percolation was introduced by Wilkinson and Willemsen, see \cite{WW}.
For a detailed study of this process
see \cite{LPS}, or the earlier works \cite{CCN}, \cite{Ale} and \cite{J2}).
To each edge $e$ we assign, independent of
the other edges, a random variable $\tau_e$, uniformly distributed
in the interval $(0,1)$. We construct, recursively, a
growing tree. Initially the tree consists only of one vertex, say
$O$. At each step we consider all edges that have exactly one
endpoint in the tree that has been created so far. From these edges
we select the one with smallest $\tau$ value and add it (and its
`external' endpoint) to the tree. Let $\tau(n)$ be the $\tau$ value
of the $n$th edge invaded by this procedure.
%It is is not hard to
%see that, for the lattices in Theorem \ref{uni-bound},
%YYY
For any infinite transitive graph $G$, it is proved in \cite{HPS} that
\begin{equation}
\label{hpseq}
\limsup_{n \rightarrow \infty} \tau(n) = p_c,
\end{equation}
where $p_c$ is the critical probability for bond percolation.
Further, note that, if all $\tau(n) < p_c$, then $O$ belongs to an
infinite cluster on which all $\tau$ values are smaller than
$p_c$. For the graphs in the statement of Theorem \ref{uni-bound}
this latter event has probability $0$. (See \cite{G} for this
classical result and references). Hence, for these lattices,
(a.s.) there is an $n$ with $\tau(n) > p_c$. This, together with
\eqref{hpseq}, implies that (a.s.) $\tau(n)$ achieves its maximum
(and that this maximum is larger than $p_c$). The following
proposition is about the invaded region at the step where this
maximum is achieved. Although this and related regions have been
under consideration before in the literature (see the subsection
`Ponds and outlets' in Stein and Newman (1995)),
this result is, as far as we know, new. \\
{\bf Remark:} The {\em invasion basin\/} of $O$ is defined similarly to the invasion tree,
except that at every step, the edge of minimal $\tau$-value among the edges outside the
current invasion basin
that have  {\em at least\/}  one endpoint in the  basin is added to the basin.
The invasion basin is typically not a tree. It is easy to
see that each edge $e$ in the invasion tree is in the invasion basin,
and  the set of sites in the invasion basin
immediately before such an edge $e$ is added to it
is the same as the set of vertices in the invasion tree immediately before $e$ is added.

\begin{prop} \label{inv-perc}
Consider invasion percolation on the square lattice (or the
triangular or the hexagonal lattice) with edge values $\tau_e$. Let
$\hat e$ be the edge with maximal $\tau$ value in the invasion basin
(as explained above). Let $\hat R$ be the radius of the region that
has been invaded up to the step where $\hat e$ is invaded. We have:
\item(a)
\begin{equation*}
 P(\hat R > n) \ge \Pc(O \leftrightarrow
\partial B(n)) ;
\end{equation*}
\smallskip
\item(b)
\begin{equation}
\label{invperc} P(\hat R > n) \approx \Pc(O \leftrightarrow
\partial B(n)), \,\,\, n \rightarrow \infty.
\end{equation}
\end{prop}

\noindent{\bf Remarks:} \\
(a) Proposition \ref{inv-perc} has triggered further research on the comparison of ponds and critical percolation
clusters: see recent refinements and generalizations in \cite{BJV}. \\
(b) The value $\hat R$ above can also be described in the
following, somewhat informal, way. Suppose each edge $e$ is closed
at time $0$ and becomes open at time $\tau_e$ (after which it
remains open). The open cluster of $O$ grows in time. Up to time
$p_c$ it is finite, but at some time larger than $p_c$ it will
become infinite (a.s). The radius of this cluster just before it
becomes infinite is $\hat R$.

\end{subsection}

\begin{subsection}{Description of the model in terms of passage times.
Autonomous regions}
Consider the description of the dynamics in the beginning of this
section, and assume for the moment that the graph is finite. Recall
that an open edge remains open and that a closed edge with at least
one green end-vertex becomes open at rate $1$. This means that if we
assign to each edge $e$ an exponentially distributed (mean $1$)
random variable $\tau(e)$, independent of the other edges (and of
the initial colours of the vertices), the time evolution of the
process can be completely described in terms of the initial colours
of the vertices and the $\tau-$ variables of the edges: Each edge
$e$ remains closed until the time $t$ at which $L_t(e)$ (defined
below) has Lebesgue measure $\tau_e$. (If no such time exists, the
edge remains closed forever). Here $L_t$ is defined by
\begin{equation}
\label{Ldef} L_t(e) = \{s < t \, : \, e \mbox{ has at least one
green end-vertex at time } s\}.
\end{equation}
(Since, once a vertex is green it can change colour only one more
time, $L_t(e)$ is clearly an interval or union of two intervals).
When $e$ becomes open and one of its end-vertices is white or red,
the appropriate action in the description in Section 1.1 is carried
out instantaneously.

In the following this equivalent description of the process turns
out to be very convenient. To illustrate it and to emphasize the
difference with one of the modified models that will be discussed in
Subsection 1.4, we give the following example:

\begin{example}
\label{dynex1}
Consider the graph with vertices denoted by $\{1, 2, 3, 4, 5\}$ and edges $\lan i, i+1 \ran$, $1 \leq i \leq 4$.
Suppose that the initial colours of the vertices $1, \cdots, 5$ are red, green, white, green, red
respectively, and that the $\tau$ values of the edges $\lan 1, 2 \ran, \cdots \lan 4, 5 \ran$ are
$6$, $3$, $4$ and $2$ respectively. As one can check by following the above description,
the initially green vertex $2$ becomes red at time $5$ due to vertex $5$.
\end{example}

\smallskip\noindent
Now suppose some finite, but possibly large, graph $G$ is given,
together with initial colours $c(v), v \in V$ and `opening times'
$\tau(e), e \in E$. Further suppose we are only interested in the
time evolution in a small subgraph of $G$, for instance just one
initially green vertex $v$. Do we need to `follow' the process in
the whole graph to reconstruct what happens at $v$? Often this is
not the case. An instructive  example is when $v$
is incident to three
edges, $e$, $e'$ and $e''$ with the properties that $\tau(e)$ is
smaller than $\tau(e')$ and $\tau(e'')$, and that the other
end-vertex of $e$, which we denote by $w$, is red. In that case we
know that $v$ is green until time $\tau(e)$ and from then on is red
(which would also happen in the `isolated' graph consisting only of
the vertices $v$ and $w$ and the edge $e$). This holds no matter
what the initial colours of the vertices in $V \setminus \{v,w\}$
and the $\tau$-values of the edges in $E \setminus \{e, e', e''\}$
are. Note that this still holds when we extend $G$ to a bigger graph
(with $c$ and $\tau$-variables) as long as we don't add extra edges
to $v$.

This brings us to the notion of {\it autonomous set}: Let
$H=(V(H),E(H))$ be a finite sub-graph of a graph $G$, and let
$\bar E$ be a finite set of external edges of $H$, i.e. edges  in
$G$, which have exactly one vertex in $V(H)$.  Assume that we have
given an initial colour assignment $c(v)$ to all $v \in V(H)$ and
opening times $\tau(e)$ to all $e \in E(H) \cup \bar E$. Let $\bar
H$ be the minimal graph containing $H$ as subgraph and $\bar E
\subset E(\bar H)$. We say that $(H,\bar E)$ is {\bf autonomous}
(with respect to $\tau$ and $c$), if for every finite subgraph
$G_0$ of $G$  which has $\bar H$ as a subgraph, the growth process
on $G_0$ starting with a colour pattern and opening times
extending the above given $c$'s and $\tau$'s has, restricted to
$H$, always the same time evolution, i.e. the same evolution as it
would have with $G_0=\bar H$, and which does not depend on colours
at the vertices in $\bar H$ not in $H$. In the simple example
considered in the previous paragraph, the graph with vertices $v$
and $w$, and edge $e$, together with the set of external edges
$\bar E = \{e', e''\}$, is autonomous.

Often, when the identity of $\bar E$ is obvious and the choice of
$c$- and $\tau$- variables is considered to be known, we simply say
that $H$ is autonomous. For this reason we might refer to the
autonomous set as ``autonomous subgraph".

\smallskip\noindent
Now suppose we have an infinite graph $G$ with given $\tau$- and $c$- variables. If every vertex (and every edge) is
contained in a finite autonomous subgraph of $G$, the infinite-volume time evolution on $G$ can be defined
in an obvious way.
The key of the proof of Theorem \ref{mainthm} is to show that, under the condition in the theorem, these autonomous
subgraphs exist with probability $1$. That is, for almost-all initial colour patterns, and almost-all $\tau$-values
each vertex and edge is contained in a finite autonomous region.

\end{subsection}

\begin{subsection}{Some alternative versions of the model}
There are many modifications or generalizations of our model (which we will sometimes call the
{\it basic model} to distinguish it from these modified versions). Below we mention four of them.

\smallskip\noindent
(i) In the basic model the $\tau$ variables are exponentially distributed. It is easy to see that if
the initial colours of the vertices are given,
and none of them is white,
the time evolution is essentially determined by the order statistics of the $\tau$ variables. It is
also easy to see that in that case
each edge $e$ becomes open at time $\tau_e$ or remains closed forever.
From such observations it easily follows that, if $p_w = 0$, replacing the exponential distribution of the
$\tau$ variables by some other continuous distribution, leaves the law of the process unchanged, apart from
an obvious time change.
This is not true if $p_w > 0$. However, as one can easily see from its proof, Theorem \ref{mainthm} remains valid
under such replacement of distribution.

\smallskip\noindent
(ii) Recall that in our basic model an edge $e$ becomes open at the
smallest time $t$ with the property that the subset of times $s < t$
at which $e$ has at least one green end-vertex, has Lebesgue measure
$\tau_e$. A natural modification of this rule is the one where
$e=\lan v, w \ran$ becomes open at the smallest time $t$ with the
property that $v$ is green throughout the interval $[t - \tau_e, t)$
or $w$ is green throughout the interval $[t - \tau_e, t)$. To
illustrate the difference between the rules, consider again the
graph with $\tau$ values and initial colours in Example
\ref{dynex1}. As can be easily checked, under the modified rule the
vertex $2$ will no longer become red due to vertex $5$ but due to
vertex $1$ (and at time $6$ instead of $5$). It turns out that
Theorem 1.1 remains valid for this modified model and that its proof
only needs some small modifications.

\smallskip\noindent
(iii) The third modification is the following model in continuous
space. Consider two homogeneous Poisson point processes $\zeta_G$,
$\zeta_R$ on $\mathbb R^d$, with intensities $\lambda_G=1$,
$\lambda_R \equiv \lambda \in (0, + \infty) $ respectively.  The
points of $\zeta_G$ ({\it green}) are interpreted as sources of
growth, and those of $\zeta_R$ ({\it red}) as sources of
``paralyzing poison''. All other elements of $\R^d$ are uncoloured.
From each source in $\zeta_G$ at time zero a green Euclidean sphere
begins to grow with constant speed 1 (of its radius). When two or
more green spheres intersect, they keep growing in the same manner,
but we say that they have become connected  (are in the same
connected green component). If a growing green sphere hits a red
region, its {\it entire} connected green component (note that this
is a union of spheres) instantaneously gets red and stops growing.
Analogs of the questions for our basic model in Subsection 1.1, in
particular the existence question, arise naturally, but so far we
have made very little progress. Although at first sight there is
some resemblance with the model studied in \cite{HaM}, the arguments
used there seem not to work here.

\smallskip\noindent
(iv)
Consider the following change of rule of the previous model (model (iii) above):
When a green sphere hits a red region, {\it only} the centers of all the spheres
of its connected green component become red; the remaining parts of the spheres
disappear (become uncoloured). This change makes the model much easier to handle
(using an invasion procedure resembling the one we will use in Section 2 for the
case $p_w =0$ of our basic model), but also
considerably less natural, and we will not discuss it in more detail.

\end{subsection}

\end{section}

\begin{section}{Proofs for the case $p_w = 0$}
\begin{subsection}{General properties for the case $p_w = 0$}
The case where $p_w = 0$
is considerably easier
than the general case and serves as
a good introduction to the latter.
We
start with some deterministic observations and claims. Let us first
restrict to a finite graph $G$, with given $\tau$-values and
$c$-values. We assume that at least one vertex has initial colour
red, at least one vertex has initial colour green, and no vertex has
initial colour white. Let $x$ be a vertex with initial colour green,
and let $t(x)$ denote the time at which $x$ becomes red. Let $\Pi$
denote the set of all paths of which the starting point is $x$ and
the end-vertex has initial colour red. It is easy to see that

\begin{equation}
\label{lowerbd}
t(x) \geq \min_{\pi \in \Pi} \max_{e \in \pi} \tau(e).
\end{equation}

Indeed, for each $t$ smaller than the r.h.s. of \eqref{lowerbd}
there is a `cut set' of edges that are still closed at time $t$
and `shield' $x$ from all initially red vertices. It is also quite
easy to see that equality holds in \eqref{lowerbd}. The
algorithmic (and inductive) argument below is not the most direct
one but has the advantage that it gives more, namely  an elegant
and suitable construction of an autonomous region. This particular
construction will almost immediately lead to a proof of parts (a)
and (b) of Theorem \ref{mainthm} for the case $p_w = 0$. The
`algorithm' is a modification (`stopped' version) of the standard
invasion percolation procedure (starting at $x$) described a few
lines above Proposition \ref{inv-perc}. At each stage of the
procedure we have a tree which is a subgraph of $G$. Initially
this tree consists only of the vertex $x$. At each step we
consider all edges that have exactly one end-vertex in the tree,
also called the {\it external edges} of the tree. Among these
edges we select the one with minimal $\tau$-value and add it (and
its external end-vertex) to the tree. The procedure is stopped as
soon as an initially red vertex is added to the tree. Let us
denote this vertex by $R$, and the final tree given by this
procedure by $T(x)$. Let $\tau^*$ be the maximal $\tau$ value on
this tree, and $e^*$ the edge where this maximum is attained.
Removing this edge from the tree $T(x)$ `splits' the tree in two
parts. Let $T_1^*(x)$ denote the part containing $x$.

\medskip\noindent
\begin{claim}
\label{claimpw0}

\smallskip\noindent
(i) The vertex $R$ is responsible for $x$ becoming red. \\
(ii) $x$ becomes red at time $\tau^*$. That is, $t(x) = \tau^*$.
Moreover, $C_g(x)$ (defined in
part (b) of the Theorem) is the set of vertices of $T_1^*(x)$. \\
(iii). Let $\bar E$ denote the set of all edges of which one
end-vertex is a vertex of $T(x)$, different from $R$, and one
end-vertex is not in $T(x)$. Let $\widehat{T}(x)$ be the graph with
the same vertices as $T(x)$ and with all edges that have both
end-vertices in $T(x)$. Then $(\widehat{T}(x), \bar E)$ is
autonomous (with respect to this coloring).
\end{claim}

\smallskip\noindent
\begin{proof} (of Claim)
The proof of the Claim is by induction on the number of steps in the
above invasion procedure. If the number of steps is $1$ we are in
the situation that the edge incident to $x$ with minimal $\tau$-
value has a red end-vertex, and the above Claim follows easily.
(Note that this case corresponds with the example in the second
paragraph of Subsection 1.3).  Now suppose the number of steps is
larger than $1$. Consider the edge $e^*$ defined above. Let $E^*$
denote the set of external edges, except $e^*$ itself, at the stage
of the procedure immediately before $e^*$ was added. From the
definition of invasion percolation, all edges in $E^*$ have
$\tau$-value larger than $\tau^*$. On the other hand, all edges that
were added after that step have, by definition, $\tau$-value smaller
than $\tau^*$. Therefore the edges in $E^*$ were never added to the
tree. Hence, since $R$ was added after $e^*$ (and was the first red
point added to the tree), it follows that every path in $G$ from $x$
to a red point contains $e^*$ or an edge in $E^*$. Therefore, by
\eqref{lowerbd} we get that
$$t(x) \geq \tau^*.$$

To get the reversed inequality, note the following. Let $y$ denote
the external end-vertex of $e^*$ when $e^*$ was added to the tree.
We already remarked that removing $e^*$ from $T(x)$ `splits'
$T(x)$ in two separate trees, and we denoted the part containing
$x$ by $T_1^*(x)$. Let $T_2^*(x)$ denote the other part. It
follows from the above that $T_2^*(x)$ contains $y$ and $R$. We
will assume that the initial colour of $y$ is green (otherwise the
Claim follows easily). It is easy to see from the above that a
similar invasion procedure as before, but now starting at $y$
instead of $x$, has as its final tree the tree $T_2^*(x)$. By the
induction hypothesis we have that $y$ becomes red at the time
which is equal to the maximal edge value in $T_2^*(x)$ and hence
before time $\tau^*$, and that $R$ is responsible for $y$ becoming
red. Also note that, from the earlier observations, just before
time $\tau^*$ there is an open path from $x$ to the end-vertex
$\neq y$ of $e^*$. Since $e^*$ becomes open at time $\tau^*$ it
follows that $x$ becomes red at time $\tau^*$. Moreover, since $R$
is responsible for $y$ becoming red, it is also responsible for
$x$ becoming red. This (and the earlier made observation that all
external edges $\neq e^*$ of $T_1^*(x)$ have $\tau$ value larger
than $\tau^*$)) completes part (i) and (ii) of the proof of
Claim~\ref{claimpw0}. Similar arguments show part (iii).
\end{proof}

\medskip\noindent
Now we are ready to handle the case where $G$ is infinite. If $G$ is
infinite and $p_r > 0$, it is not a priori clear that the process
described in Subsection 1.1 is well-defined. However, the above
invasion procedure and the corresponding Claim now give us the
instrument to define it and to give a proof of Theorem \ref{mainthm}
in this particular case.
\end{subsection}

\begin{subsection}{Proof of Theorem \ref{mainthm} for the case $p_w =
0$} For each green vertex $x$ simply run the invasion procedure
starting from $x$. Since the initial colours and the $\tau$
variables are independent, we have, at each step in the invasion
from $x$, probability $p_r$ of hitting a red vertex (independently
of the previous steps in this invasion). Hence the invasion
procedure starting at $x$ stops with probability $1$, and (by part
(iii) of Claim~\ref{claimpw0}) yields an autonomous region
containing $x$. Since the graph has countably many vertices, this
yields a construction of the process on $G$ and completes the
proof of part (a) of the theorem. Moreover it shows that
Claim~\ref{claimpw0} also holds (a.s.) for $G$. Further, the
number of steps in the invasion procedure from an initially green
vertex clearly has a geometric distribution: the probability that
it is larger than $n$ is $(1-p_r)^n$. Since (by part (ii) of
Claim~\ref{claimpw0}) $|C_g(v)|$ is at most the number of steps in
the invasion procedure,
part (b) of the theorem follows. \\
{\it Proof of part (c)}: For each pair of vertices $x,y$, let
$I(x,y)$ denote the event that $x$ is initially green and that $y$
is initially red and responsible for $x$ becoming red. It follows
immediately from the above that for all vertices $x$ and all $m$
\begin{equation}
\label{Rbd} \sum_{y : d(x,y) \geq m} P(I(x,y)) \, = \, P\left(d(x,
R(x)) \geq m \right) \le (1-p_r)^m.
\end{equation}
Further, using that $G$ is a Cayley graph, the `mass transport
principle' (see e.g. Section 7.1 in \cite{LyP} or \cite{HPS}) gives:
$$ P\left(D(w) \mbox{has radius } \geq m\right) \leq \sum_{v \,: \, d(v,w) \geq m} \! \! P(I(v,w)) =
\sum_{v \,: \, d(v,w) \geq m} \! \! P(I(w,v)),$$ which by
\eqref{Rbd} is at most $(1-p_r)^m$. This completes the proof of
part (c) of the theorem.

\smallskip\noindent
{\it Proof of part (d)}. As we will see, this follows from earlier
observations, together with a block argument which is quite
similar to one in percolation theory, due to Kesten (see
\cite{K}). Below we denote the $d-$dimensional cubic lattice
simply by $\Z^d$.

Let, as before, $T(x)$ denote the tree produced by the invasion
procedure starting at $x$. We want to prove exponential decay for
$P(|D(v)| > n)$, where $v$ is an initially red point. Without loss
of generality we take $v = {\bf 0}$. We say that a finite set $W$ of
vertices containing $\bf 0$ is a lattice animal (abbreviated as
l.a.) if for all $w \in W$ there is a path in $\Z^d$ from $\bf 0$ to
$w$ of which every vertex is in $W$. From the definitions (and
since, as we saw in (c), $D(\bf 0)$ is a.s. finite), it is clear
that $D(\bf 0)$ is a lattice animal. Let $L$ be an even integer and
consider the partition of $\mathbb{Z}^d$ into cubes $Q_L (x) :=
[-L/2, L/2 )^d + L x$, $ x \in \Z^d$. We say that $x \in \Z^d$ is
{\it fine} if $Q_L(x) \cap D({\bf 0}) \neq \emptyset$. Let $V_F$
denote the set of all vertices that are fine. Since $D({\bf 0})$ is
a lattice animal, $V_F$ is also a lattice animal. Further, we say
that $x \in \Z^d$ is {\it proper} if $Q_L(x)$ contains a vertex $y$
with $|T(y)| > L/4$, and write $I( x \mbox{ is proper })$ for the
indicator function of the corresponding event. (Here $T(\cdot)$ is
as defined in the invasion procedure earlier in this Section).
Finally, a subset of $\Z^d$ is proper if every element in the set is
proper. It is clear that for every $x \neq \bf 0$, if $x$ is fine,
then $x$ is proper. It is also clear that if $D(\bf 0)$
contains vertices outside $Q_L(\bf 0)$, then $\bf 0$ is also proper.
Recall from Claim~\ref{claimpw0}(iii) that for each tree $T$ in $\Z^d$ and each vertex $y$, the
event $\{T(y) = T\}$ depends only on the $c$ values of the vertices
of $T$ and the $\tau$ values of the edges that have at least one
end-vertex in $T$. From this it easily follows that the process
$\left(I( x \mbox{ is proper }), \,  x \in \Z^d\right)$ is
$2$-dependent (see e.g. \cite{G} for this notion). Let $\ep=\ep(L) = \ep(L,d)$ be the probability that a
given vertex is proper. Since, for each $y$,  the distribution of
$|T(y)|$ is geometric (and $|Q_L(y)|$ is polynomially bounded in
$L$) it is clear that for fixed $d$
$$\ep(L,d) \ra 0 \mbox{ as } L \ra \infty.$$
The above mentioned $2$-dependence gives that there is a constant $C_1 = C_1(d)$ such that
for every set $W \subset \Z^d$

\begin{equation}
\label{1dp} P(W \mbox{ is proper }) \leq \ep^{\frac{|W|}{C_1}}.
\end{equation}

Finally, we use that there is a constant $C_2 = C_2(d)$ such that
the number of lattice animals of size $m$ is at most $C_2^m$, see~\cite{G}.
Together, the above gives that (noting that each l.a. of size $\geq
m$ contains a l.a. of size $m$) that for $n$ large enough (depending on $L$),
\begin{eqnarray}
& & P(|D({\bf 0})| > n) \leq P\left(\exists \mbox{ a proper l.a.
of size } \lceil\frac{n}{|Q_L|}\rceil \right)  \\ \nonumber \leq & &
C_2^{\frac{n}{|Q_L|}+1} \, \ep(L)^{\frac{n}{|Q_L| C_1}} \\
\nonumber = & & C_2 \, \left[ \left(C_2
\,\,\ep(L)^{\frac{1}{C_1}}\right)^{1/Q_L}\right]^n.
\end{eqnarray}
Taking $L$ so large that $C_2 \,\, \ep(L)^{(1/C_1)} < 1$ completes
the proof of part (d). This completes the proof of Theorem
\ref{mainthm} for the special case where $p_w = 0$.
\end{subsection}

\begin{subsection}{Proof of Proposition \ref{inv-perc} and Theorem \ref{uni-bound}}
We first prove Proposition \ref{inv-perc}. We say that an edge is
$p$-open if $\tau_e < p$. Define $p$-open paths and $p$-open
clusters in the obvious way. To prove the Proposition we will derive
suitable lower and upper bounds for the
l.h.s. of \eqref{invperc} in terms of an expression of the form of its r.h.s. \\
The lower bound is very easy: Since $\tau_{\hat e} > p_c$ (see the
paragraph below \eqref{hpseq}), it follows immediately that (a.s)
the region which is already invaded at the step where $\hat e$ is
invaded, contains all the vertices of the $p_c$-open cluster
of $O$. Hence the l.h.s of \eqref{invperc} is larger than or equal to the r.h.s.\\
The upper bound is more complicated.
We use the standard percolation notation $\theta(p)$ for the probability that $O$ is in
an infinite $p$-open cluster.\\
Define, for each $p$ and $n$, the following two events:

\begin{eqnarray*}
A_{n,p} = \{\exists \mbox{ a } p \mbox{-closed circuit with
diameter}\geq n &\mbox{in the dual lattice}\\\mbox{ that contains }
O \mbox{ in its interior}\}.
\end{eqnarray*}

$$D_p = \{O \mbox{ belongs to an infinite } p \mbox{-open cluster }\}.$$
Note that $P(D_p) = \theta(p)$ and that if $p_1 < p_2$,
then $D_{p_1} \subset D_{p_2}$ and $A_{n, p_2} \subset A_{n, p_1}$.

Let $\hat \tau = \tau_{\hat e}$.
Let $p'$ be some number between $p_c$ and $1$.
The following observation is straightforward.

\smallskip\noindent
{\it Observation}\\
(a) If $\hat \tau > p'$ and $\hat R  \geq n$, then there is a $p > p'$
such that the event $A_{n,p}$ occurs. \\
(b) If $\hat \tau <p'$, then there is a $p<p'$ such that $D_p$ occurs.

\medskip\noindent
Let, for $ p > p_c$, $L(p)$ be the correlation length
(=$L(p,\ep_0)$) as defined in Section 1 in the paper by Kesten
(1987) on scaling relations. (See \cite{K2}). That is, $L(p)$ is
the smallest $n$ such that the probability that there is a
$p$--open horizontal crossing of a given $n \times n$ box is
larger than $1 - \ep_0$. Here $\ep_0$ is an appropriately
(sufficiently small) chosen positive constant. (From this
definition it is clear that $L(p)$ is non-increasing in $p$ on the
interval $(p_c,1]$). It is well-known (see (2.25) in \cite{K2} and
the references preceding that equation) that there are constants
$C_1 > 0$ and $C_2
> 0$ such that for all $p >p_c$ and all $n$,

\begin{equation} \label{ke1}
P_p(A_{n,p}) \leq C_1 \, \exp\left(- \frac{C_2 n}{L(p)}\right).
\end{equation}

Further, Theorem 2 in \cite{K2} says that there is a constant $C_3
> 0$ such that, for all $p > p_c$,

\begin{equation} \label{ke2}
\theta(p) \leq C_3 \Pc\left(O \leftrightarrow \partial B(L(p))\right).
\end{equation}

Now take, for $p'$, the supremum of those $p$ for which $L(p) >
n/(C_4 \log n)$, where $C_4$ is a positive constant that will be
appropriately chosen later.
Obviously,
\begin{equation}
\label{obv-eq}
P(\hat R \geq n) \leq P(\hat R \geq n, \, \hat \tau > p')
+ P(\hat \tau < p').
\end{equation}
The first term in the r.h.s of \eqref{obv-eq} is, by Observation
(a) above and the `nesting' property of the events $A_{n,p}$
(stated in the sentence below the definition of these events),
smaller than or equal to
\begin{equation}
\label{obv1} \lim_{p \downarrow p'} P(A_{n,p}) \leq \limsup_{p
\downarrow p'} C_1 \exp(- \frac{C_2 n}{L(p)}) \leq C_1 \exp(- C_2
C_4 \log n),
\end{equation}
where the first inequality follows from \eqref{ke1} and the second inequality from the definition of
$p'$.

The second term of \eqref{obv-eq} is, by Observation (b) and the
`nesting' property of the events $D_p$, smaller than or equal to
\begin{equation}
\label{obv2}
\lim_{p \uparrow p'} \theta(p) \leq \limsup_{p \uparrow p'} C_3 \Pc\left(O \leftrightarrow \partial B(L(p))\right)
\leq C_3 \Pc\left(O \leftrightarrow \partial B(\frac{n}{C_4 \log n})\right),
\end{equation}
where the first inequality follows from \eqref{ke2} and the second follows by (again) using the
definition of $p'$.
Putting \eqref{obv-eq}, \eqref{obv1} and \eqref{obv2} together we have

\begin{equation} \label{kec}
P(\hat R \geq n)
\leq C_3 \Pc\left(O \leftrightarrow \partial B(\frac{n}{C_4 \log n})\right)
+ C_1 \, \exp(- C_2 C_4 \log n).
\end{equation}

It is believed that $\Pc(O \leftrightarrow \partial B(n))$ has a
power law behaviour. This has only been proved for site percolation
on the triangular lattice. However, for the percolation models we
are considering, we do know that this function of $n$ has power-law
lower and upper bounds. As a consequence we can choose $C_4$ so
large that the second term in the r.h.s. of \eqref{kec} is (for all
large enough n) smaller than the first term. Finally, it follows
quite easily from RSW arguments (see e.g. Sections 11.7 and 11.8 in \cite{G})
that $\Pc\left(O \leftrightarrow
\partial B(n/C_4\log n )\right)\approx\Pc\left(O \leftrightarrow
\partial B(n)\right)$. This completes the proof of Proposition
\ref{inv-perc}. $\Box$

\smallskip\noindent
Now we are ready to prove Theorem \ref{uni-bound}.  The invasion
procedure in Subsection 2.1, which was used in the proof of
Theorem \ref{mainthm}, differs from the `ordinary' invasion
percolation model (described in the paragraphs preceding
Proposition \ref{inv-perc}, in that is stops as soon as the growing
tree `hits' a red vertex. There is strictly speaking another
difference: the $\tau$ values in the former case were
exponentially distributed and those in the latter case were
uniformly distributed on $(0,1)$. However, that difference clearly
doesn't matter, and in the rest of this proof we assume the $\tau$
variables to be uniformly distributed on $(0,1)$. Let us call the
former procedure a `stopped' invasion procedure (with parameter
$p_r$), and the latter an ordinary invasion procedure. All these
procedures (the stopped procedures with $p_r$ varying between $0$
and $1$, and the ordinary procedure) can be coupled in the
following natural way: Assign to each vertex $v$, independent of
the others, (and of the $\tau$ variables) a random variable
$\rho(v)$, uniformly distributed on the interval $(0,1)$. When we
now do invasion percolation (w.r.t. the $\tau$ variables) and stop
when we hit a vertex with $\rho$ value smaller than $p_r$, this
corresponds exactly with the above mentioned stopped invasion with
parameter $p_r$. In this coupled setting, the set $C_g(O)$ for the
stopped model with parameter $p_r$ is clearly non-increasing in
$p_r$, and the union of these sets over all the values $p_r>0$ is
exactly the region mentioned in Proposition \ref{inv-perc}. Theorem
\ref{uni-bound} now follows from this proposition. $\Box$

\end{subsection}

\end{section}

\begin{section}{Proof for the case $p_w > 0$}
In this section we prove Theorem \ref{mainthm} for the case $p_w>0$.
Recall that in the
special case where there are no white vertices (see Section 2) there
was an elegant invasion procedure which produced, with probability
$1$, a finite autonomous set containing a given vertex or edge. This
is much more complicated in the general case, when there are white
vertices. We still have a procedure which, if it stops, gives an
autonomous set containing, say, a given vertex $x$. This algorithm
starts as before, with one invasion tree, which initially consists
only of the vertex $x$, and which grows by invading the edge with
minimal $\tau$ value. However, when we hit a `fresh' white vertex
$y$ we have to investigate the `space-time paths from outside' that
have possibly influenced $y$. This is done by starting new invasion
trees in the green vertices on the boundary of the white cluster of
$y$. As before, an invasion tree stops when it invades a red vertex.
In the situation in the previous Section this also marked the end of
the algorithm. But in the current situation it only marks the end of
one invasion tree, while the others keep growing and creating new
invasion trees. In this way the algorithm might go on forever.
However, we show that under the condition in Theorem \ref{mainthm} the
algorithm, which is described more precisely below, does end.

\smallskip\noindent
The input is a connected graph $G = (V,E)$, the initial colours $c(v), v \in
V$ and the opening times $\tau(e), e \in E$, and the vertex $x$ or
edge $e$ for which we want to find an autonomous region.  Here we
only handle the case concerning a vertex $x$ and we assume that $x$
is green; the other cases can be done in a very similar way. For the
moment it suffices to restrict to finite graphs. The algorithm will
produce an autonomous subgraph $H$ and, for some vertices $v$ of
$H$,  non-negative numbers $t_g(v)$ and $t_r(v)$, and for some edges
$e$ of $H$ a positive number $t(e)$. Here $t_g(v)$ and $t_r(v)$ will
denote the time at which $v$ becomes green and red, respectively.
The value $t(e)$ will be the time when $e$ becomes open. It will
be clear from the description below that, at each stage of the
algorithm the edges to which a $t$-value has been assigned form a
collection of disjoint trees. Each tree in this collection has one of two labels:
`active' or `paralyzing'.
How these labels are assigned is described in Subsection
\ref{Desc} below.
The collection of active trees is
denoted by $\calT_a$ and the collection of paralyzing trees by $\calT_p$.
As we will see, new active or paralyzing trees are `created' during
the algorithm, and active trees can merge with each other or
with a paralyzing tree. In the former case the new tree is active, in the latter case
it is paralyzing.

The set of edges which have at least one
end-vertex in an active tree (and not both end-vertices in the same
active tree) is denoted by $\calE$. With some abuse of terminology
we say that a vertex is in $\calT_a$ if it is a vertex of some tree
in $\calT_a$. A similar remark holds w.r.t. $\calT_p$.

Apart from the above, we need the following auxiliary variables and structures,
which will be assigned during the algorithm.

The first auxiliary structure we mention here is a set $S$, which
can be interpreted as the set of all initially white vertices that `have been
seen until the current stage' in the algorithm.
We say that a vertex `is registered' if it is in $\calT_p$, $\calT_a$ or $S$.
Further, to each
edge $e \in \calE$ (as introduced above) a value $t_1(e)$ will be
assigned, which can be interpreted as a tentative, possible value
for $t(e)$.

Finally, the following definition will be important:
The {\it white cluster} $C_w(v)$ of a vertex $v$ is defined as the maximal connected subset of $G$ of
which all vertices $y$ have initial colour $c(y) = $ white. (Note that this notion, in contrast with
the notion of green clusters (defined in
Section 1) does not involve the state (open/closed) of the edges.
The boundary of the white cluster of $v$, denoted by $\partial C_w(v)$, is the set of all vertices that
are not in $C_w(v)$ but have an edge to some vertex in $C_w(v)$.
If $c(v)$ is not white, then $C_w(v)$ and $\partial C_w(v)$ are empty.

\begin{subsection}{Description of the algorithm} \label{Desc}
Using the notions above we are now ready to describe the algorithm. It starts with action 1 below,
followed by an iteration of (some of) the other actions. Recall that $c(x)$ is green.

\medskip\noindent
{\bf 1.} {\bf Initialization of some of the variables and structures.} \\
Set $\calT_p = \emptyset$, $\calT_a = \{\{x\}\}$, and $S = \emptyset$. \\
Set $t_g(x) = 0$, $\calE $ as the set of all edges incident to $x$,
and $t_1(e) = \tau(e)$ for all edges $e\in \calE$.

\smallskip\noindent
{\bf 2.} {\bf Selection of minimal external edge.} \\
Remove from $\calE$ all edges of which both endpoints are in the same tree of $\calT_a$. \\
{\it Comment: such edges can have resulted from some af the actions below} \\
If $\calE = \emptyset$, stop.
Otherwise, let $e$ be the edge in $\calE$ with minimal $t_1$-value. \\
Write $e = \lan v,y \ran$ with $v$ in $\calT_a$. (This way of
writing is of course not unique if both end-vertices
of $e$ are in $\calT_a$ but that doesn't matter). Let $T$ denote the tree in $\calT_a$ of which $v$ is a vertex. \\
If $y$ is not in $\calT_a$, $\calT_p$ or $S$ (that is, $y$ is `fresh') go to 2a, else go to 2b.

\smallskip\noindent
{\bf 2a. Fresh vertex.} \\
Determine $c(y)$. \\
If $c(y) =$ red, set $t(e) = t_1(e)$ and go to 3a. \\
If $c(y) =$ green, set $t(e) = t_1(e)$ and go to 4. \\
If $c(y) =$ white, go to 6.

\smallskip\noindent
{\bf 2b. Registered vertex.} \\
Set $t(e) = t_1(e)$. \\
If $y$ is in $\calT_p$ go to 3b. \\
If $y$ is in $\calT_a$ go to 5. \\
Else go to 7.

\medskip\noindent
{\bf 3a. Fresh red.} \\
{\it Comment: This case can be handled in almost the same way as 3b
below and therefore, with an `administrative trick',
we simply turn it into the latter case:} \\
Set $t_r(y) = 0$.
Add to $\calT_p$ the tree which consists only of the vertex $y$. \\
Go to 3b.

\medskip\noindent
{\bf 3b.  Active tree $T$ becomes paralyzed.}
Set $t_r(z) = t(e)$ for all vertices $z$ of $T$. \\
Remove from $\calE$ all edges of which one end-vertex is in $T$
and the other end-vertex is not in $\calT_a$. Let $T'$ be the tree
in $\calT_p$ of which $y$ is a vertex. Replace, in $\calT_p$, the
tree $T'$ by that obtained from `glueing together' $T$ and $T'$
via the edge $e$.
Remove $T$ from $\calT_a$. \\
Go to 2.

\smallskip\noindent
{\bf 4. Fresh green. } \\
Set $t_g(y) = 0$. For each edge $e'$ incident to $y$ that was not
yet in $\calE$: add $e'$ to $\calE$ and set $t_1(e') = \tau(e')$.
Replace, in $\calT_a$, the tree $T$ by a new tree obtained from
glueing $y$ to $T$ by the edge $e$. \\
%Remove $e$ from $\calE$. \\
Go to 2.

\smallskip\noindent
{\bf 5. Two active trees join.} \\
Let $T' \in \calT_a$ be the active tree of which $y$ is a vertex.
Replace, in $\calT_a$, the trees $T$ and $T'$ by a new tree
obtained from `glueing together' $T$ and $T'$ with the edge $e$. \\
%Remove $e$, and all other edges
%between $T$ and $T'$, from $\calE$. \\
Go to 2.

\smallskip\noindent
{\bf 6. Fresh white.} \\
Add every vertex of $C_w(y)$ to $S$. \\
For each vertex $z$ in $\partial C_w(y)$ that has $c(z) =$ green and is not
in $\calT_a$ or $\calT_p$, do the following: \\
Set $t_g(z) = 0$; add the tree $\{z\}$ to $\calT_a$; add  to $\calE$
each edge $e'$ incident to $z$ that is not yet in $\calE$, and set
$t_1(e') = \tau(e')$.

\smallskip\noindent
For each vertex $z$ in $\partial C_w(y)$ that has $c(z) =$ red and is not
in $\calT_p$, set $t_r(z) = 0$ and add the tree $\{z\}$ to $\calT_p$.\\
Go to 2.

\smallskip\noindent
{\bf 7. Registered white.} \\
Set $t_g(y) = t(e)$. Replace, in $\calT_a$, the tree $T$ by the
tree obtained from $T$ by `glueing' the vertex $y$ to it by the
edge $e$. For each edge $e' = \lan y, z \ran$ of $y$ that is not
in
$\calE$, add it to $\calE$ and set $t_1(e')$ as follows: \\
If $z$ is in $\calT_p$ but $c(z) \neq$ red, set
\begin{equation}
\label{adjust} t_1(e') = t(e) + \tau(e') -(t_r(z) - t_g(z)),
\end{equation}
else set
$$t_1(e') = t(e) + \tau(e').$$
{\it Comment: The subtracted term in \eqref{adjust} accounts for the time that $e'$ already had
a green end-vertex. See also the Remark at the end of Subsection 3.2} \\
Go to 2.

\medskip\noindent
{\bf Remark:} \\
Note that initially there is only one active tree and that new
active trees are only formed in part 6 of the algorithm. Also note
that initially there are no paralyzing trees; these can be formed in
part 6 and in part 3a. Moreover, 3a always leads, via 3b, to the
elimination of an active tree. Now consider the case that $G$ has no
vertices with initial colour white. Then the algorithm never enters
part 6 (neither part 7) so that throughout the algorithm there is
one active tree until a red vertex is `hit'. From such
considerations it is easily seen that in this case the algorithm
reduces to the one described in Section 2.
\end{subsection}

\begin{subsection}{Correctness of the algorithm}
If $G$ is finite the above algorithm will clearly stop. Moreover, we
claim that if $G$ has at least one vertex with initial colour red,
we have the following situation at the end of the algorithm: The set
of active trees $\calT_a$ is empty. The set $\calT_p$ contains one
or more trees, and the vertex $x$ is in one of them. Each of these
trees has exactly one vertex with initial colour red, and this
vertex is `responsible' for  the other vertices in that tree to
become red. The following pair, $(H, \bar E)$, is autonomous: The
vertices of $H$ are the vertices in $\calT_p$ together with all
vertices in $S$. The edges of $H$ are all edges of which both
end-vertices are in the above set. The set $\bar E$ is the set of
all edges of which one end-vertex is a vertex
$v$ of $H$ with $c(v) \neq$ red, and the other end-vertex is not in $H$.
Further, each initially green vertex $v$ of $H$ becomes red at
time $t_r(v)$.

The `correctness' of the above algorithm (that is,
the above claim) can, in principle, be proved by induction,
e.g. on the number of edges. Instead of giving a full proof (which would be extremely tedious)
we present the key ideas/observations ((a) - (d) below) to be used in such proof.

\smallskip\noindent (a) As in many induction proofs it is useful, or even necessary, (for carrying out
the induction step)  to generalize
the statement one wants to prove. In the current situation this generalization is as follows: In the above algorithm,
information is stored in the administration when the vertices involved are `encoutered' by
the algorithm. In particular, in action 6 a white cluster and its boundary are
`stored' because a vertex of the white cluster had been encountered (as endpoint of the edge selected in action 2).
The same algorithm still works if at one or more stages of the algorithm such information about a white cluster (and its
boundary) is stored `spontaneously' (that is, without this cluster having been encoutered in the sense above).

\smallskip\noindent
(b) The main observation for doing induction on the number of edges is the following:
Let, among all edges with at least one initially
green endpoint, $\hat e$ be the one with minimal $\tau$ value. Let $\hat x$ and $\hat y$ denote its endpoints.
We may assume that $\hat x$ is initially green. It is clear that the first thing that happens in the `real' growth
process is the opening of $\hat e$ (namely, at time $\tau(\hat e)$). It is alo clear that from that moment
on the growth process behaves
as if starting on a graph with one vertex less, namely the graph obtained by `identifying' (or glueing together) $\hat x$
and $\hat y$ (with an obviously assigned colour: green if $c(y)$ is white or green; red if $c(y)$ is red).

\smallskip\noindent
(c) To carry out the induction step it has to be shown that the algorithm has a property analogous to
that for the real
process described in (b) above. That this is indeed the case, can be seen as follows: As long as $\hat x$ and $\hat y$
are not `registered' in the algorithm, the algorithm behaves the same as it would behave
for the graph obtained after the identification described in (b).
Moreover, one can easily see from the description of the algorithm that immediately after one of these vertices is
registered,
the other one also is, and that they are immediately `attached to each other' (by the edge $\hat e$) in the same
tree.

\smallskip\noindent
(d) The following side remark must be added to (c) above: Suppose
that $\hat y \in C_w(y)$ in action 6 at some stage of the
algorithm.  This cluster $C_w(y)$ could be larger than that in the
graph obtained by identifying $\hat y$ and $\hat x$. This means
that in that step `more information is collected' than in the
situation where $\hat x$ and $\hat y$ would be identified from the
beginning. It is exactly for this issue that the generalized
algorithm (and claim) in (a) was given.
\end{subsection}

\begin{subsection}{Proof of Theorem \ref{mainthm}}
\begin{proof}
 It follows, in the same way as in the case $p_w = 0$, that on
an infinite graph the dynamics is well-defined provided the
algorithm stops with probability $1$. We will show that, under the
condition \eqref{key} in the statement of the Theorem, the algorithm
indeed stops. In fact, the arguments we use will give something
stronger, namely Proposition \ref{prop1} below, from which not only
part (a) of Theorem~\ref{mainthm} follows, but which we will also use to prove
part (b), (c) and (d).
\begin{prop}\label{prop1}
Under the condition of Theorem~\ref{mainthm}, we have that, for each $x$, the
above mentioned algorithm stops, and, moreover, the distributions of
the volume and the diameter of the graph $H$ defined above have an
exponential tail.
\end{prop}
\begin{proof} By the $k$th step of the algorithm we mean everything done by the
algorithm between the $k$th and $k+1$th time the algorithm
`enters' part 2a in the description in Subsection 3.1. Recall that
we say that a vertex is registered if it is in $\calT_a$,
$\calT_p$ or $S$. Let $\nu_k$ be the number of registered vertices
at the beginning of step $k$. (In particular, $\nu_1 = 1$.) If the
algorithm is already terminated during step $j$ for some $j < k$,
we set $\nu_k$ equal to the number of registered vertices at the
moment of termination. Further, let $y_k$ denote the `fresh'
vertex (i.e. the vertex $y$ in part 2a of the description in
Subsection 3.1) treated in step $k$ of the algorithm. (In
particular, $y_1$ is the end-vertex of the edge incident to $x$
with minimal $\tau$ value). Let $\eta_k = \nu_{k+1} - \nu_k$.
%be the increase of 'inspected vertices' by step $k$ of
%the algorithm.
Further, let $\al_k$ denote the net increase of the number of active trees during step $k$ of the algorithm.
If the algorithm is terminated during step $k$, we set $\al_k = -1$.
(This choice is somewhat arbitrary; it is simply a suitable choice to ensure that certain statements below
hold for all $k$).

Note that the initial colours of the vertices are independent
random variables, each being white, red or green with probability
$p_w$, $p_r$ and $p_g$ respectively. It is clear from the
algorithm that we may consider the colour of a vertex as `hidden'
until the moment the vertex becomes registered. Let $\calF_k$ be
all information obtained by the algorithm until the beginning of
step $k$
(including the identity but not the colour of $y_k$).\\
Let
$N = \min\{n \, : \, 1 + \sum_{k=1}^n \alpha_k = 0\}$.
It is easy to see that if $N$ is finite the algorithm stops during or before step $N$, and the number
of vertices in the above defined graph $H$ is
\begin{equation}
\label{Hbd}
1 + \sum_{k=1}^N \eta_k.
\end{equation}

Note that if $c(y_k)$ is white, the procedure is sent to part 6, and
the newly registered vertices in step $k$ of the algorithm are
exactly the vertices of $C_w(y_k)$ and the not yet registered
vertices of $\partial C_w(y_k)$; moreover, $|\calT_a|$ increases
during this step by at most the number of green vertices in
$\partial C_w(y_k)$. We write {\it at most}, because during the
remainder of step $k$ no new active trees are created but already
present active trees may disappear (which happens if the algorithm
enters part 3b before it enters part 2a again.

Similarly, if $c(y_k)$ is red or green, then the only newly registered vertex is $y_k$ itself; moreover,
in the former case
$|\calT_a|$ goes down during step $k$ by at least $1$, while in the latter case it goes down or
doesn't change. \\
For every connected set $W$ of vertices with $|W| \ge 2$, the number of vertices in
the boundary of $W$ is at most $(D-1) |W|$; hence, we have

\begin{equation}
\label{etbd}
\eta_k \leq D |C_w(y_k)| + \I_{\{c(y_k) \mbox{ not white}\}}.
\end{equation}
\begin{equation}
\label{albd}
\al_k \leq (D-1) |C_w(y_k)| - \I_{\{c(y_k) \mbox{ is red}\}}.
\end{equation}

Note that (since $y_k$ is `fresh') the conditional probability that $c(y_k)$ is red, white
or green, given $\calF_k$, is
$p_r$, $p_w$ and $p_g$
respectively. Also note that, by the condition in the Theorem, $p_w < 1/(D-1)$ and hence
(as is well-known and easy to check) there is a $q <1$ such that for all $n$ and all vertices $v$,

\begin{equation}
\label{expcw}
P(|C_w(v)| \geq n) \leq q^n.
\end{equation}

Moreover, it is easy to see that conditioned on $\calF_k$, which includes the information that $y_k$ is
a specific vertex, say $y$, the cluster size
$|C_w(y_k)|$ is stochastically smaller than $|C_w(y)|$.
Hence the bound \eqref{expcw} also holds (a.s) if we replace its l.h.s. by $P(|C_w(y_k)| \ge n | \calF_k)$.
This, combined with \eqref{etbd}
%and \eqref{albd}
immediately gives that there is a
$\ga < 1$ such that for all $k$ and $n$,

\begin{equation}
\label{etdbd}
P(\eta_k \geq n | \calF_k) \leq \ga^n.
\end{equation}

As to the $\alpha$'s, define (compare \eqref{albd}), for every vertex $v$,
\begin{equation}
\label{alvbd}
\al(v) = (D-1) |C_w(v)| - \I_{\{c(v) \mbox{ is red}\}}.
\end{equation}

Let $\al'(v), \, v \in V$ be independent copies of the $\alpha(v), \, v \in V$.
By a similar stochastic domination argument that led to \eqref{etdbd}, we have
for all vertices $v$, and all
positive integers $k$ and $n$,

\begin{equation}
\label{aldbd}
P(\alpha_k \geq n | \calF_k,\, y_k = v) \leq P(\alpha(v) \geq n) = P(\alpha'(v) \geq n).
\end{equation}

And, again by \eqref{expcw}, there is a $\la < 1$ such that for all $n$ and $v$

\begin{equation}
\label{alv-xbd}
P(\alpha'(v) \geq n) = P(\alpha(v) \geq n) \leq \lambda^n.
\end{equation}

Further note that, for each vertex $v$, we have $E(|C_w(v)|) = \xi_v(p_w)$.
Hence, condition \eqref{key} in Theorem 1.1 says that there is an
$\ep > 0$ such that for all vertices $v$ we have
%$(D-1) E(|C_w(v)|)
%- p_r < -\ep$. This and \eqref{albd} (and again the above mentioned
%stochastic domination) give

\begin{equation}
\label{almbd}
E(\al'(v)) = E(\al(v)) < -\ep.
\end{equation}

\noindent
From \eqref{aldbd} and the definition of the random variables $\al'(v), \, v \in V$,
it follows (from stochastic domination) that, for all positive integers $K$,

\begin{equation}
\label{supbd}
P\left(\sum_{k=1}^K \alpha_k \geq 0\right) \leq
\sup^* P\left(\sum_{k=1}^K \alpha'(v_k) \geq 0\right),
\end{equation}
where we use '*' to indicate that the supremum is taken over all tuples
of $K$ distinct vertices $v_1, v_2, ..., v_K$.

\noindent
From \eqref{alv-xbd} and \eqref{almbd}
it follows (by
standard large-deviation upper bounds
for independent random variables)
that there is a $\beta < 1$ such that for all $K$ and all distinct vertices
$v_1, v_2, ..., v_k$,

$$
P(\sum_{k=1}^K \alpha'(v_k)  \geq 0) \leq \beta^K.
$$

\noindent
From this and \eqref{supbd} it follows that the distribution of $N$
has an exponential tail.

%From \eqref{aldbd} and \eqref{almbd} it follows (by standard
%arguments) that $N$ (which was defined a few lines before
%\eqref{Hbd}) is (a.s.) finite and that its distribution has an
%exponential tail.
Putting this together with  \eqref{etdbd} and \eqref{Hbd} we  that
the number of vertices  in $H$ has an exponential tail. Indeed the
event that $1+\sum_{k=1}^N \eta_k \ge n$ is contained in the union
of the events $N \ge an$ and $\sum_{k=1}^{an} \eta_k \ge n$; the
probabilities of these events decay exponentially in $n$ for
suitable $a$.

This completes the proof of Proposition \ref{prop1}. (Note that the diameter of $H$ is
at most its volume, since $H$ is a connected graph).
\end{proof}

\smallskip\noindent
{\it Parts (a) and (b) of Theorem~\ref{mainthm}} follow immediately
from Proposition~\ref{prop1} (noting that the vertices of $C_g(x)$ belong to $H$). \\
Using Proposition~\ref{prop1}, {\it Parts (c) and (d)} of the Theorem~\ref{mainthm} can now be derived in the same way as in the special
case $p_w=0$ in Section 2.
This completes the proof of Theorem \ref{mainthm}.
\end{proof}

\medskip\noindent
{\bf Remark:} For the alternative model (i) in Subsection 1.4, the
proof of Theorem \eqref{mainthm} is exactly the same. Note that the
proof doesn't use that the $\tau's$ are exponentially distributed,
it applies in the same manner to any continuous
distribution. \\
For the alternative model (ii) the algorithm in
Subsection 3.1
needs a few small adaptations. Apart from this the proof remains practically the same.

\end{subsection}
\end{section}
{\bf Acknowledgments.} Two of the authors (V.S. and M.E.V.) learned
about the continuum model from E.J. Neves. We thank Antal J\'{a}rai
for comments on Proposition \ref{inv-perc} and Chuck Newman for
drawing our attention to the article \cite{StN}. We also thank Ron
Peled and the referees for corrections in the first manuscript.

\end{document}